\documentclass[11pt,titlepage]{article}
\usepackage{amsfonts}
\usepackage{amsmath}
\usepackage{amssymb}
\usepackage{amsthm}
\usepackage{fullpage}
\usepackage{setspace}
\usepackage{subfigure}
\usepackage{graphicx}
\usepackage{url}
\usepackage[numbers]{natbib}
\usepackage{multirow}
\usepackage{url}
\usepackage{hyperref}
\usepackage[utf8]{inputenc}
\usepackage[T1]{fontenc}
\usepackage{lmodern}
\usepackage{comment}

\usepackage{fancyhdr}
\newtheorem{theorem}{Theorem}[section]
\newtheorem{example}[theorem]{Example}

\newtheorem{lemma}[theorem]{Lemma}
\newtheorem{corollary}[theorem]{Corollary}
\newtheorem{conjecture}[theorem]{Conjecture}
\newtheorem{remark}[theorem]{Remark}
\newtheorem{question}[theorem]{Question}

\newenvironment{changemargin}[2]{
\begin{list}{}{
 \setlength{\topsep}{0pt}
 \setlength{\leftmargin}{#1}
 \setlength{\rightmargin}{#2}
 \setlength{\listparindent}{\parindent}
 \setlength{\itemindent}{\parindent}
 \setlength{\parsep}{\parskip}
}
\item[]}{\end{list}}

\begin{document}


\pagenumbering{arabic}

\begin{center}
\textsc{\Large Patterns in the Coefficients of Powers of Polynomials Over a Finite Field}\\[.5cm]
\textsc{Kevin Garbe}\\[1cm]
\end{center}

\begin{changemargin}{1.5cm}{1.5cm}
\small \textsc{Abstract}. We examine the behavior of the coefficients of powers of polynomials over a finite field of prime order. Extending the work of Allouche-Berthe, 1997, we study  $a(n)$, the number of occurring strings of length $n$ among coefficients of any power of a polynomial $f$ reduced modulo a prime $p$. The sequence of line complexity $a(n)$ is $p$-regular in the sense of Allouche-Shalit. For $f=1+x$ and general $p$, we derive a recursion relation for $a(n)$ then find a new formula for the generating function for $a(n)$. We use the generating function to compute the asymptotics of $a(n)/n^2$ as $n\to\infty$, which is an explicitly computable piecewise quadratic in $x$ with $n=\lfloor p^m/x\rfloor$ and $x$ is a real number between $1/p$ and 1. Analyzing other cases, we form a conjecture about the generating function for general $a(n)$. We examine the matrix $B$ associated with $f$ and $p$ used to compute the count of a coefficient, which applies to the theory of linear cellular automata and fractals. For $p=2$ and polynomials of small degree we compute the largest positive eigenvalue, $\lambda$, of $B$, related to the fractal dimension $d$ of the corresponding fractal by $d=\log_2(\lambda)$. We find proofs and make a number of conjectures for some bounds on $\lambda$ and upper bounds on its degree.
\end{changemargin}
\normalsize
\section{Introduction}

It was shown by S. Wolfram and others in 1980s that 1-dimensional linear cellular automata lead at large scale to interesting examples of fractals. 
A basic example is the automaton associated to a polynomial $f$ over $\Bbb Z/p$, whose transition matrix $T_f$ is the matrix of multiplication by $f(x)$ on the space of Laurent polynomials in $x$. 
If $f=1+x$, then starting with the initial state $g_0(x)=1$, one recovers Pascal's triangle mod $p$.
For $p=2$, at large scale, it produces the Sierpinski triangle shown in Figure \ref{fig:1x2}. Similarly, the case of $f=1+x+x^2$, $p=2$, and initial state $g_0(x)=1$ produces the fractal shown in Figure \ref{fig:1xx22}.

The double sequences produced by such automata, i.e., the sequences encoding the coefficients of the powers of $f$, have a very interesting structure. 
Namely, if $p$ is a prime, they are $p$-automatic sequences in the sense of \cite{AS1}. 
In the case $f=1+x$, this follows from Lucas' theorem that $\binom{n}{k}=\prod\limits_i\binom{n_i}{k_i}$ mod $p$, where $n_i,k_i$ are the $p$-ary digits of $n,k$. 

In \cite{W1,W2}, S. Wilson studied this example in the case where $f$ is any polynomial, and computed the fractal dimension of the corresponding fractal. 
The answer is $\beta=\log_p(\lambda)$, where $p\le \lambda\le p^2$ is the largest (Perron-Frobenius) eigenvalue of a certain integer matrix $B$ associated to $f$ (in particular, an algebraic integer). 
In terms of coefficients of powers of $f$, this number characterizes the rate of growth of the total number of nonzero coefficients in $f^i$ for $0\le i<p^n$: this number behaves like $n^\beta$. 
The number of nonzero coefficients of each kind can actually be computed exactly at every step of the recursion, by using a matrix method similar to Wilson's; this is explained in the paper \cite{AS1}. 

In this paper, we compute the eigenvalues $\lambda$ and their degrees for $p=2$ for Laurent polynomials $f$ of small degrees, observe some patterns, and make a number of conjectures (in particular, that $\lambda$ can be arbitrarily close to $4$) in Section \ref{sec:Eigen}. 
We also prove an upper bound for $\lambda$ depending on the degree of $f$. 

The size of the matrix $B$ (which is an upper bound for the degree of $\lambda$) is the number of accessible blocks (i.e., strings that occur in the sequence of coefficients of $f^i$ for some $i$) of length $deg(f)$ (for $p=2$). 
This raises the question of finding the number $a(n)$ of accessible blocks of any length $n$. 
The number $a(n)$ characterizes the so-called line complexity of the corresponding linear automaton, and is studied in the paper \cite{AB}. 
It is shown in \cite{AB},\cite{B}, and references therein that $C_1n^2\le a(n)\le C_2n^2$, and that for $p=2$ and $f=1+x$, one has $a(n)=n^2-n+2$. 
More generally, however, the sequence $a(n)$ does not have such a simple form, even for $f=1+x$ and $p>2$. 
The paper \cite{AB} derives a recursion for this sequence, and we derive another one in Section \ref{sec:Recur1xp}, which is equivalent. 
These recursions show that the sequence $a(n)$ is $p$-regular in the sense of \cite{AS2} (the notion of $p$-regularity is a generalization of the notion of $p$-automaticity, to the case of integer, rather than mod $p$, values). 
We then proceed to find a new formula for the generating function for $a(n)$ in Section \ref{sec:Closed}, and use it to compute the asymptotics of $a(n)/n^2$ as $n\to \infty$ in Section \ref{sec:Limits}. 
It turns out that if $n=\lfloor{p^m/x}\rfloor$, where $x$ is a real number between $1/p$ and $1$, then $f(n)/n^2$ tends to an explicit function of $x$, which is piecewise quadratic (a gluing together of 3 quadratic functions, which we explicitly compute). 
In Section \ref{sec:Limits} we also compute the maximum and minimum value of this function, which gives the best asymptotic values for $C_1$ and $C_2$. 
This gives us new precise results about the complexity of the Pascal triangle mod $p$. 
We also perform a similar analysis for $f=1+x+x^2$ and $p=2$, and make a conjecture about the general case.

\begin{figure}[h]
\begin{center}
\includegraphics[width=.7\textwidth]{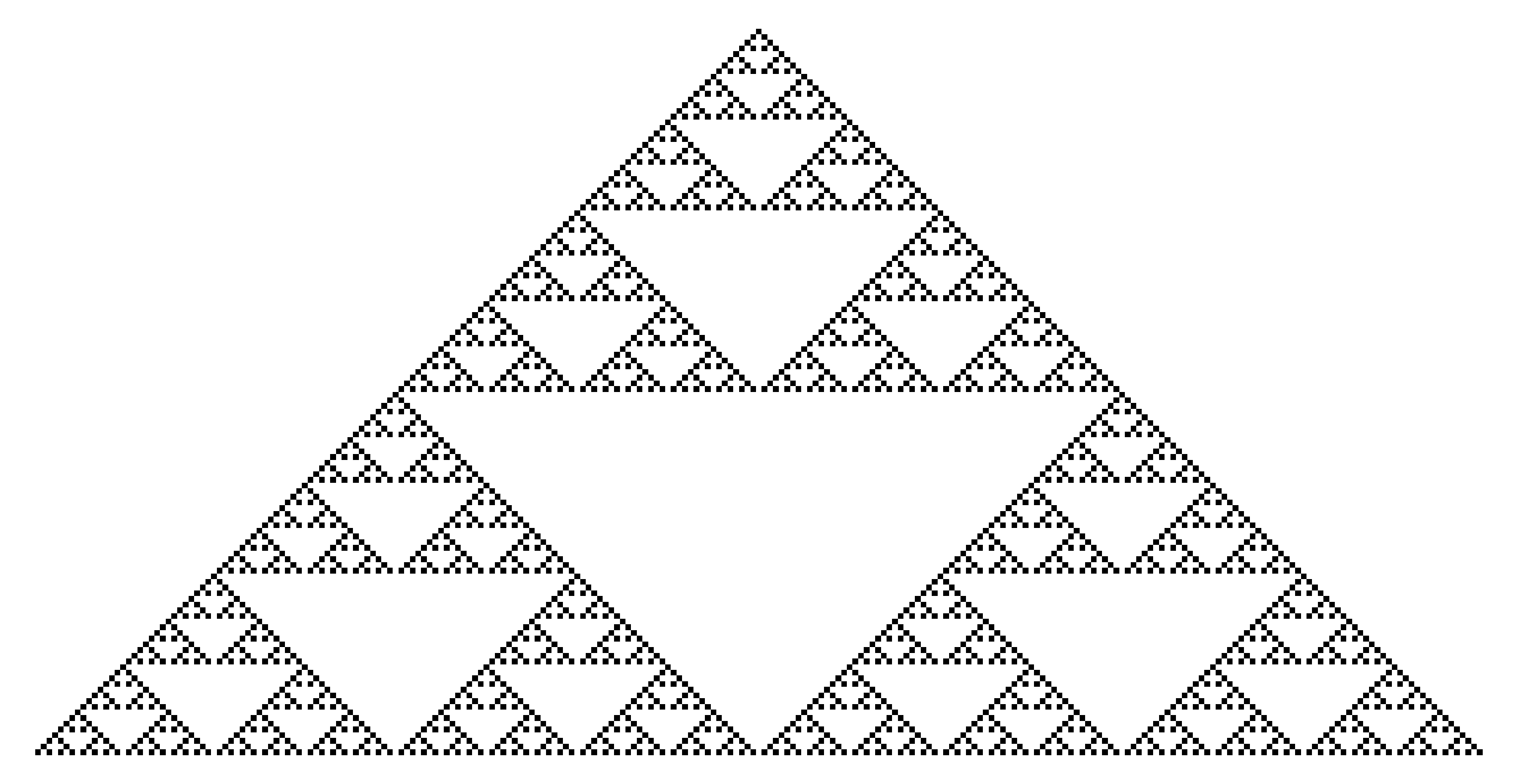}
\caption{Fractal corresponding to $1+x$ modulo 2 (Sierpinski's Triangle)}
\label{fig:1x2}
\includegraphics[width=.7\textwidth]{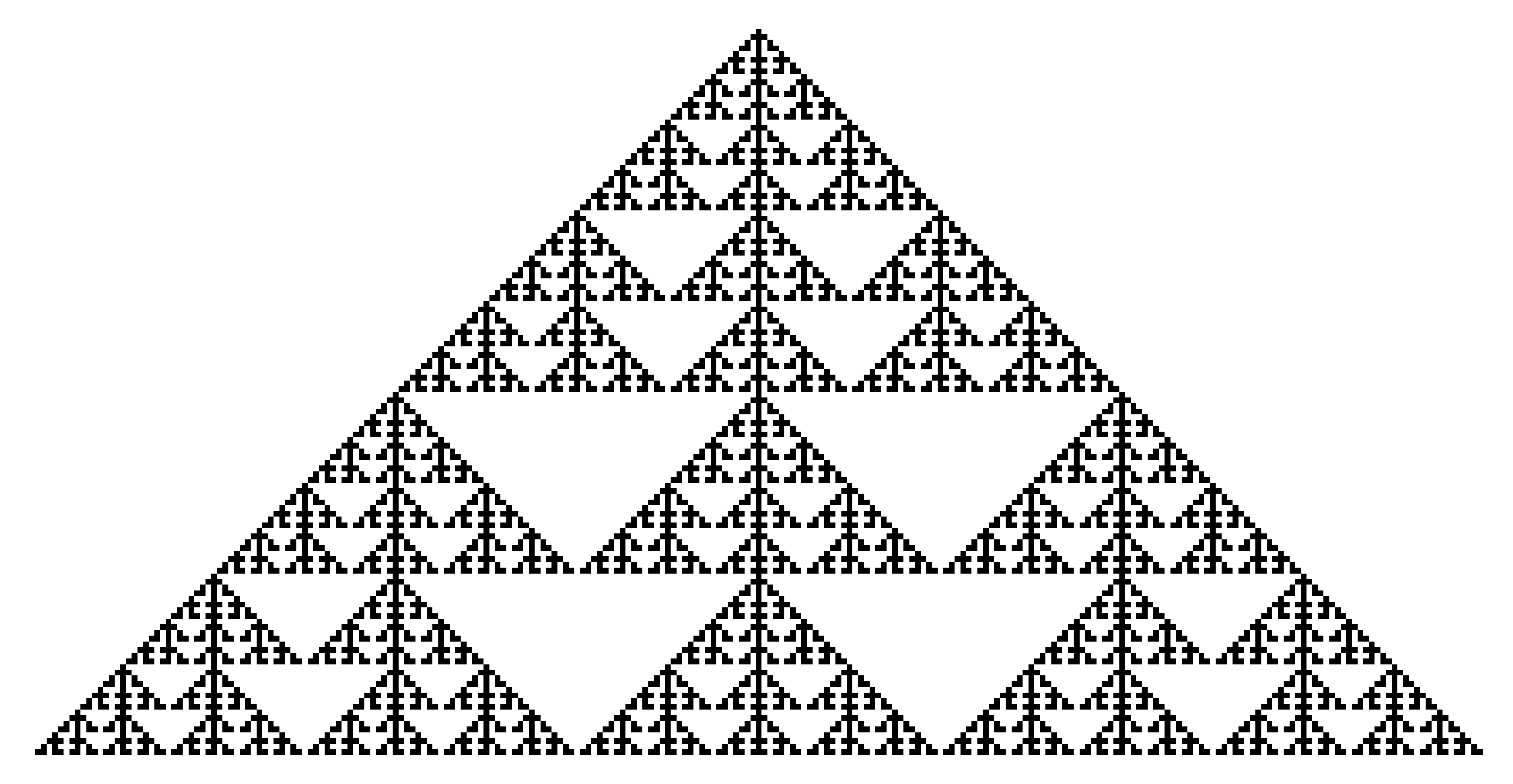}
\caption{Fractal corresponding to $1+x+x^2$ modulo 2}
\label{fig:1xx22}
\end{center}
\end{figure}

\section{Accessible Blocks}
\label{sec:Blocks}

\subsection{Definitions}
\label{sec:BlockDef}

A block is a string of mod $p$ digits. An $m$-block is a block with $m$ digits. For example, the four $2$-blocks modulo $2$ are $00,01,11,$ and $11$.

For a polynomial $f(x)$ with integer coefficients reduced modulo $p$, an accessible $m$-block is an $m$-block that appears anywhere among the coefficients, ordered by powers of x, of powers of $f(x)$ modulo $p$. 
The number of accessible $0$-blocks we define to be $1$.
Furthermore, we define row $k$ for some $f(x)$ and $p$ to be the coefficients of $f(x)^k$ reduced modulo $p$ and define $a_{f(x),p}(m)$ to be the number of accessible $m$-blocks for the polynomial $f(x)$ and prime $p$. 

\begin{example}
For $f(x)=1+x$ and $p=2$, the $4$-blocks $1101$ and $1011$ are never a substring of any power of $1+x$ reduced modulo $2$. Every other $4$-block appears in some power of $1+x$ reduced modulo $2$, so $a_{1+x,2}(4)=14$.
\end{example}

\subsection{Recursion Relations for $a(n)$}
\label{sec:Recursion}

We start with the well known fact in Lemma \ref{lem:multipower}.
\begin{lemma}
\label{lem:multipower}
$f(x)^{k\cdot p} \equiv f(x^p)^k \pmod{p}\text{.}$
\end{lemma}

Applying Lemma \ref{lem:multipower} to the accessible blocks, we have Corollary \ref{cor:zerospace}.

\begin{corollary}
\label{cor:zerospace}
For any integer $k$, prime $p$, and polynomial $f(x)$, every row $k\cdot p$ for $f(x)$ mod $p$ is of the form $b_10\ldots0 b_20\ldots\ldots 0b_{n-1}0\ldots 0b_n$ where the entries $b_i$ are the coefficients of $f(x)^k$, and where each string of zeros between two entries $b_i$ and $b_{i+1}$ is of length $p-1$. 
Therefore, every accessible block from a row divisible by $p$ is a subsection of $b_10\ldots0 b_20\ldots\ldots 0b_{n-1}0\ldots 0b_n$.
\end{corollary}

\subsubsection{Accessible $m$-Blocks for $f(x)=1+x$ and General Prime $p$}
\label{sec:Recur1xp}

The number of accessible $m$-blocks for $f(x)=1+x$ and any prime $p$, $a_{1+x,p}$, is defined by the recurrence relation in Theorem \ref{thm:general1x}.

\begin{theorem}
\label{thm:general1x}
For $f(x)=1+x$ and any prime $p\ge3$, for $0\leq k \leq p-1$, the recursion relation with starting points $a_{1+x,p}(0)=1, a_{1+x,p}(1)=p, \text{ and } a_{1+x,p}(2)=p^2$ is
\begin{align*}
a_{1+x,p}(p \cdot n+k)=&\frac{(p-k)(p-k+1)}{2} \cdot a_{1+x,p}(n)+(kp+k-k^2+\frac{p^2-p}{2}) \cdot a_{1+x,p}(n+1)\\&+\frac{k^2-k}{2} \cdot a_{1+x,p}(n+2)-(2p-1)(2p-2).
\end{align*}
\end{theorem}

\begin{proof}
From Corollary \ref{cor:zerospace}, every accessible block in a row $r$ with $r \equiv 0 \pmod{p}$ is formed by adding $p-1$ zeros between every digit of an accessible block, then adding some number of zeros (possibly none) less than $p$ to either side. 
Furthermore, because $f(x)=1+x$, the coefficient of $x^i$ in a row is the sum modulo $p$ of the coefficients of $x^i$ and $x^{i-1}$ in the previous row. 
Because accessible blocks are subsections of a row, any accessible $m$-block comes from an accessible $(m+1)$-block. 
Table \ref{tbl:epictable} provides the general forms of the $(p\cdot n+k)$-blocks for each row modulo $p$. 
To count the multiple additions of $b$ in the forms, we define $g_i={p-1 \choose i}$.

\begin{table}[h]
\begin{center}
\begin{tabular}{|c|c|c|c|c|c|} \hline
&\multicolumn{5}{|c|}{Blocks for $k=$}\\\hline
\begin{tabular}{c}Row\\mod\\$p$\end{tabular} &0&1&2&$\cdots$&$p-1$
\\\hline  0 &
\begin{tabular}{*7{@{}c@{}}}
$b_1000$&$\ldots$&$00b_200$&$\ldots\ldots$&$00b_n00$&$\ldots$&$000$\\
$0b_100$&$\ldots$&$000b_20$&$\ldots\ldots$&$000b_n0$&$\ldots$&$000$\\
$00b_10$&$\ldots$&$0000b_2$&$\ldots\ldots$&$0000b_n$&$\ldots$&$000$\\
$\vdots$&$\vdots$&$\vdots$&$\ddots$&$\vdots$&$\vdots$&$\vdots$\\
$0000$&$\ldots$&$b_10000$&$\ldots\ldots$&$b_{n-1}0000$&$\ldots$&$0b_n0$\\
$0000$&$\ldots$&$0b_1000$&$\ldots\ldots$&$0b_{n-1}000$&$\ldots$&$00b_n$
\end{tabular}
&\begin{tabular}{@{}c@{}}$b_{n+1}$\\0\\0\\$\vdots$\\0\\0\end{tabular}
&\begin{tabular}{@{}c@{}}0\\$b_{n+1}$\\0\\$\vdots$\\0\\0\end{tabular}
&\begin{tabular}{@{}c@{}}$\cdots$\\$\cdots$\\$\cdots$\\$\ddots$\\$\cdots$\\$\cdots$\end{tabular}
&\begin{tabular}{@{}c@{}}0\\0\\0\\$\vdots$\\$b_{n+1}$\\0\end{tabular}
\\\hline  1 &
\begin{tabular}{*7{@{}c@{}}}
$b_1000$&$\ldots$&$0b_2b_200$&$\ldots\ldots$&$0b_nb_n00$&$\ldots$&$00b_{n+1}$\\
$b_1b_100$&$\ldots$&$00b_2b_20$&$\ldots\ldots$&$00b_nb_n0$&$\ldots$&$000$\\
$0b_1b_10$&$\ldots$&$000b_2b_2$&$\ldots\ldots$&$000b_nb_n$&$\ldots$&$000$\\
$\vdots$&$\vdots$&$\vdots$&$\ddots$&$\vdots$&$\vdots$&$\vdots$\\
$0000$&$\ldots$&$b_10000$&$\ldots\ldots$&$b_{n-1}0000$&$\ldots$&$b_n0$\\
$0000$&$\ldots$&$b_1b_1000$&$\ldots\ldots$&$b_{n-1}b_{n-1}000$&$\ldots$&$b_nb_n$
\end{tabular}
&\begin{tabular}{@{}c@{}}$b_{n+1}$\\$b_{n+1}$\\0\\$\vdots$\\0\\0\end{tabular}
&\begin{tabular}{@{}c@{}}0\\$b_{n+1}$\\$b_{n+1}$\\$\vdots$\\0\\0\end{tabular}
&\begin{tabular}{@{}c@{}}$\cdots$\\$\cdots$\\$\cdots$\\$\ddots$\\$\cdots$\\$\cdots$\end{tabular}
&\begin{tabular}{@{}c@{}}0\\0\\0\\$\vdots$\\$b_{n+1}$\\$b_{n+1}$\end{tabular}
\\\hline $\vdots$ &$\vdots$&$\vdots$&$\vdots$&$\ddots$&$\vdots$
\\\hline $p-1$&
\begin{tabular}{*3{@{}c@{}}}
$b_1b_2(g_2b_2)$&$\ldots \ldots$&$(g_4b_{n+1})(g_3b_{n+1})(g_2b_{n+1})$\\
$(g_2b_1)b_1b_2$&$\ldots \ldots$&$(g_5b_{n+1})(g_4b_{n+1})(g_3b_{n+1})$\\
$(g_3b_1)(g_2b_1)b_1$&$\ldots \ldots$&$(g_6b_{n+1})(g_5b_{n+1})(g_4b_{n+1})$\\
$\vdots$&$\ddots$&$\vdots$\\
$(g_2b_1)(g_3b_1)(g_4b_1)$&$\ldots \ldots$&$(g_2b_n)b_nb_{n+1}$\\
$b_1(g_2b_1)(g_3b_1)$&$\ldots \ldots$&$(g_3b_n)(g_2b_n)b_n$
\end{tabular}
&\begin{tabular}{@{}c@{}}$b_{n+1}$\\$(g_2b_{n+1})$\\$(g_3b_{n+1})$\\$\vdots$\\$(g_2b_{n+1})$\\$b_{n+1}$\end{tabular}
&\begin{tabular}{@{}c@{}}$b_{n+2}$\\$b_{n+1}$\\$(g_2b_{n+1})$\\$\vdots$\\$(g_3b_{n+1})$\\$(g_2b_{n+1})$\end{tabular}
&\begin{tabular}{@{}c@{}}$\cdots$\\$\cdots$\\$\cdots$\\$\ddots$\\$\cdots$\\$\cdots$\end{tabular}
&\begin{tabular}{@{}c@{}}$(g_3b_{n+2})$\\$(g_4b_{n+2})$\\$(g_5b_{n+2})$\\$\vdots$\\$b_{n+1}$\\$(g_2b_{n+1})$\end{tabular}
\\\hline
\end{tabular}
\end{center}
\caption{Forms of blocks for the general case $1+x$ with any prime $p$}
\label{tbl:epictable}
\end{table}

The number of accessible blocks that lead into each form in Table \ref{tbl:epictable} are the triangular numbers counting downwards for $a_{1+x,p}(n)$, the triangular numbers counting upward for $a_{1+x,2}(n+2)$, and because the total number of forms is $p^2$, we find $a_{1+x,p}(n+1)$ through subtraction. 
Namely, the factor of $a_{1+x,p}(n)$ starts at $p$ for $row$ congruent to $0$ modulo $p$ and $k$=0, and decreases as $k$ and $row$ increase, and the coefficient of $a_{1+x,p}(n+2)$ starts at $0$ for $row$ congruent to $0$ and $1$ modulo $p$ and increases with $k$ and $row$. 
An additional $(2p-1)(2p-2)$ must be subtracted to account for blocks that satisfy multiple forms. 
Therefore \begin{align*}a_{1+x,p}(p \cdot n+k)=&\frac{(p-k)(p-k+1)}{2}\cdot a_{1+x,p}(n)+(kp+k-k^2+\frac{p^2-p}{2}) \cdot a_{1+x,p}(n+1)\\&+\frac{k^2-k}{2} \cdot a_{1+x,p}(n+2)-(2p-1)(2p-2).\end{align*}
\end{proof}

This is equivalent to Theorem 5.10 of Allouche-Berthe \cite{AB}, reproduced below in Theorem \ref{thm:AB}.
\begin{theorem}
\label{thm:AB}
For $0\leq k \leq p-1$ and $n \geq 0$ such that $pn+k \geq 3$
\begin{align*}a_{1+x,2}(p n+k +1)-a_{1+x,2}(p n+k)=&(p-k)\Big(a_{1+x,2}(n+1)-a_{1+x,2}(n)\Big)
\\&+k\Big(a_{1+x,2}(n+2)-a_{1+x,2}(n+1)\Big)\end{align*}
with starting points $a_{1+x,2}(0)=1$, $a_{1+x,2}(1)=p$, $a_{1+x,2}(2)=p^2$, and $a_{1+x,2}(3)=\frac{p^3+4p^2-5p+2}{2}$.
\end{theorem}

\subsubsection{Accessible $m$-Blocks for  $c+x+x^2$ and prime $p$}
\label{sec:RecurMore}

Table \ref{tbl:cxx^2data} provides $a_{c+x+x^2,p}(n)$ for small $n$ and $p$.

\begin{table}
\begin{center}
\begin{tabular}{|c|c|c|}
\hline
Prime&$c$&a(n)\\
\hline
2&1&\begin{tabular}{llllllllll}2&4&8&4&25&36&53&70&92&114\end{tabular}\\\hline
3&1&\begin{tabular}{llllllllll}3&9&25&43&71&109&157&207&259&313\end{tabular}\\\hline
3&2&\begin{tabular}{llllllllll}3&9&25&61&105&165&233&321&417&533\end{tabular}\\\hline
5&1&\begin{tabular}{llllllllll}5&25&121&393&673&929&1257&1761&2341&3097\end{tabular}\\\hline
5&2&\begin{tabular}{llllllllll}5&25&125&393&689&953&1293&1801&2389&3145\end{tabular}\\\hline
5&3&\begin{tabular}{llllllllll}5&25&117&385&657&905&1221&1713&2277&3017\end{tabular}\\\hline
5&4&\begin{tabular}{llllllllll}5&25&101&169&253&353&509&721&989&1313\end{tabular}\\\hline
7&1&\begin{tabular}{lllllll}7&49&331&1285&2137&2881&3859\end{tabular}\\\hline
\end{tabular}
\end{center}
\caption{$a(n)$ for $c+x+x^2$}
\label{tbl:cxx^2data}
\end{table}

Using a method similar to the one we used for Theorem \ref{thm:general1x}, the recursion relations appear to be those shown in Table \ref{tbl:cxx2recur}.

\begin{table}
\begin{center}
\begin{tabular}{|c|c|c|c|c|}
\hline
$p$&$c$&Recursion&k&initial\\\hline
2&1&\begin{tabular}{c}2a(n)+2a(n+1)\\a(n)+2a(n+1)+a(n+2)\end{tabular}&8&1,2,4,8,14,25\\\hline
3&1&\begin{tabular}{c}6a(n)+3a(n+1)\\3a(n)+6a(n+1)\\a(n)+7a(n+1)+a(n+2)\end{tabular}&20&1,3,9,25\\\hline
3&2&\begin{tabular}{c}4a(n)+4a(n+1)+a(n+2)\\2a(n)+5a(n+1)+2a(n+2)\\a(n)+4a(n+1)+4a(n+2)\end{tabular}&32&1,3,9,25,61,105\\\hline
5&1&\begin{tabular}{c}
9a(n)+12a(n+1)+4a(n+2)\\6a(n)+13a(n+1)+6a(n+2)\\4a(n)+12a(n+1)+9a(n+2)\\2a(n)+10a(n+1)+12a(n+2)+a(n+3)\\a(n)+12a(n+1)+10a(n+2)+2a(n+3)
\end{tabular}&152&1,5,25,121,393,673\\\hline
5&2&\begin{tabular}{c}
9a(n)+12a(n+1)+4a(n+2)\\6a(n)+13a(n+1)+6a(n+2)\\4a(n)+12a(n+1)+9a(n+2)\\2a(n)+10a(n+1)+12a(n+2)+a(n+3)\\a(n)+12a(n+1)+10a(n+2)+2a(n+3)
\end{tabular}&152&1,5,25,125,393,689\\\hline
5&3&\begin{tabular}{c}
9a(n)+12a(n+1)+4a(n+2)\\6a(n)+13a(n+1)+6a(n+2)\\4a(n)+12a(n+1)+9a(n+2)\\2a(n)+10a(n+1)+12a(n+2)+a(n+3)\\a(n)+12a(n+1)+10a(n+2)+2a(n+3)
\end{tabular}&152&1,5,25,117,385,657\\\hline
5&4&\begin{tabular}{c}
15a(n)+10a(n+1)\\10a(n)+15a(n+1)\\6a(n)+18a(n+1)+a(n+2)\\3a(n)+19a(n+1)+3a(n+2)\\a(n)+18a(n+1)+6a(n+2)
\end{tabular}&72&1,5,25,101,169\\\hline
\end{tabular}
\end{center}
\caption{Recursions for $c+x+x^2$}
\label{tbl:cxx2recur}
\end{table}

We see that for $p>2$, $a_{c+x+x^2,p}(n)=a_{1+x,p}(n)$ if $c=\frac14 \pmod{p}$ because $c+x+x^2=(1+x/2)^2$. Furthermore, we arrive at Conjecture 
\ref{conj:cxx2similar}.

\begin{conjecture}
\label{conj:cxx2similar}
For $c\neq\frac14\pmod{5}$, the recursion for $a_{1+x+x^2,p}(n)$ is independent of $c$. Only the initial terms of the recursion depend on $c$.
\end{conjecture}

\subsection{Closed form for $a(n)$}
\label{sec:Closed}

\begin{theorem}
\label{thm:closedsimple}
$a_{1+x,2}(m)=m^2-m+2$.
\end{theorem}

\begin{proof}
Theorem \ref{thm:general1x} provides the recursion relation of $a_{1+x,2}(2n)=3a_{1+x,2}(n)+a_{1+x,2}(n+1)-6$ and $a_{1+x,2}(n)=a_{1+x,2}(n)+3a_{1+x,2}(n+1)$. 
We can find the starting points of $a_{1+x,2}(1)=2$ and $a_{1+x,2}(2)=4$ through inspection. 
This uniquely defines the sequence of accessible $m$-blocks. 
It is easy to show that the equation $a_{1+x}(m)=m^2-m+2$ satisfies both recursion relations through substitution, and also satisfies $a_{1+x,2}(1)=2$ and $a_{1+x,2}(2)=4$.  
\end{proof}

This matches Remark 5.14 of \cite{AB}.

\subsubsection{Generating Functions for a(n)}
\label{sec:Generating}

Using recursion relations, we can find the generating functions $g_{f(x),p}$ for $p\geq3$.

\begin{theorem}
\label{thm:gen1xp}
\begin{align*}g_{1+x,p}(z)=&\sum\limits_{n=0}^{\infty}a_{1+x,p}(n)z^n\\=&\frac{1}{(1-z)^3}\bigg(1+(p-3)z+(p^2-3p+3)z^2\\&+z^2\frac{(p-1)^2}{2}\sum\limits_{i\geq0}\Big(pz^{p^i}-2(p-1)z^{2 p^i}+(p-2)z^{3 p^i}\Big)\bigg).\end{align*}
\end{theorem}

\begin{proof}
We have from Theorem \ref{thm:general1x} that for starting points 
$a(0)=1$, $a(1)=p$, and $a(2)=p^2$
the recursion relation is defined for $pn+k>2$ as
\begin{align*}a(pn+k)=&\frac{(p-k)(p-k+1)}{2}a(n)+(kp+k-k^2+\frac{p^2-p}{2})a(n+1)\\&+\frac{k^2-k}{2}a(n+2)-(2p-1)(2p-2).\end{align*} 
Adjusting for the $k=0,1$ cases by replacing $k$ with $n+2$ gives 
\begin{align*}a(pn+k+2)=&\frac{(p-k-2)(p-k-1)}{2}a(n)+(kp-3k-k^2-2+\frac{p^2+3p}{2})a(n+1)\\&+\frac{(k+1)(k+2)}{2}a(n+2)-(2p-1)(2p-2). \end{align*}

To adjust for the case when $p,k=0$, we define the recursion relation to have an additional term of $\frac{(p-2)(p-1)}{2}a(0)+\frac{(p-4)(p+1)}{2}a(1)-(2p-1)(2p-2)$ subtracted from the right hand side for only the case of $p,k=0$.

We multiply through by $z^{pn+k}$, then sum over $k=0$ to $p-1$, then $n=0$ to $\infty$. 
We also subtract from the right hand side of the sum the above mentioned additional term to account for the case of $p,k=0$. 
Defining $h(x)=\sum\limits_{n\geq0}a(n+2)z^n$, we get 
\begin{align*}h(z)=&(1+z+z^2+\ldots+z^{p-1})^3h(z^p)+\frac{1}{2(1-z)^3}\bigg(p^3z(1-z)^2+2p^2(1-z)(4-5z+2z^2)\\&+2(2-3z+3z^2-z^3-z^p)-p(12-19z+16z^2-5z^3-6z^p+2z^{2p})\bigg)\\&-\frac{(2p-1) (2p-2)}{1-z}.\end{align*}

Therefore $\displaystyle h(z)=\frac{(1-z^p)^3}{(1-z)^3}h(z^p)+Q(z)-(2p-1)(2p-2)\frac{1}{1-z}$ where \begin{align*} Q(z)=&\frac{1}{2(1-z)^3}\bigg( p^3z(1-z)^2+2p^2(1-z)(4-5z+2z^2)+2(2-3z+3z^2-z^3-z^p)\\&-p(12-19z+16z^2-5z^3-6z^p+2z^{2p})\bigg).\end{align*}

We then define $u(z)=(1-z)^3h(z)$ and $R(z)=Q(z)(1-z)^3-(2p-1)(2p-2)(1-z)^2$. Iteratively substituting gives
$u(z)=u(z^{p^\infty})+\sum\limits_{i\geq0}R(z^{p^i})=a(2)+\sum\limits_{i\geq0}R(z^{p^i})$, or
$h(z)=\frac{1}{(1-z)^3}\Big(a(2)+\sum\limits_{i\geq0}R(z^{p^i})\Big)$. 
Note that 
\begin{align*} \sum\limits_{i\geq0}R(z^{p^i})=&\sum\limits_{i\geq0}\frac12\bigg((p^3-2p^2-5p+2)z-2(p^3-3p^2+2p-1)z^2\\&+(p-2)(p-1)^2z^3+2(3p-1)z^p-2pz^{2p}\bigg)\\=& - \Big((3p-1)z-pz^2\Big)+\frac{(p-1)^2}{2}\sum\limits_{i\geq0}\Big(pz^{p^i}-2(p-1)z^{2 p^i}+(p-2)z^{3 p^i}\Big).\end{align*} 

Therefore
\begin{align*}
g(z)=&a(0)+a(1)z+z^2h(z)
\\=&1+pz+z^2\frac{p^2+\sum\limits_{i\geq0}R(z^{p^i})}{(1-z)^3}
\\[0.2em]=&\frac{1+(p-3)z+(p^2-3p+3)z^2+z^2\frac{(p-1)^2}{2}\sum\limits_{i\geq0}\Big(pz^{p^i}-2(p-1)z^{2 p^i}+(p-2)z^{3 p^i}\Big)}{(1-z)^3}.
\end{align*}
\end{proof}

\begin{example}
Setting $p=3$ in Theorem \ref{thm:gen1xp} and noting that the $z^{3p^i}$ further reduces when $p=3$ provides $$g_{1+x,3}(z)=\frac{1}{(1-z)^3} \Big(1+3z^2-2z^3+8z^2\sum\limits_{i=0}^\infty(z^{3^i}-z^{2\cdot3^i})\Big).$$
\end{example}

\begin{example}
Setting $p=5$ in Theorem \ref{thm:gen1xp} provides $$g_{1+x,5}(z)=\frac{1}{(1-z)^3} \Big(1+2z+13z^2+8z^2\sum\limits_{i=0}^\infty(5z^{5^i}-8z^{2\cdot 5^i}+3z^{3 \cdot 5^i})\Big).$$
\end{example}

We can use a similar proof to find further generating functions $g_{x),p}(z)$ from the recursion relations for $a_{f(x),p}(n)$.

\begin{theorem}
\label{thm:gen1xx22}
$$g_{1+x+x^2,2}(z)=\frac{1+2z^3+2z^5-z^6+\sum\limits_{i=0}^\infty(z^{2^i}-z^{3\cdot 2^i})}{(1-z^2)(1-z)^2}.$$
\end{theorem}

Based on the recursions in Table \ref{tbl:cxx2recur} and the method provided in Theorem \ref{thm:gen1xp}, we arrive at Conjecture \ref{conj:cxx2func}, which is confirmed for $p=3,5$.

\begin{conjecture}
\label{conj:cxx2func}
For $c\neq\frac14\pmod{p}$, the functional equation for the generating function $g_{c+x+x^2,p}(z)$ is $$g_{c+x+x^2,p}(z)=\frac{r(z^p)}{r(z)}g_{c+x+x^2,p}(z^p)-Q(z)-\frac{k}{1-z},$$ where $r(z)=(1-z^2)(1-z)^2$ and $Q(z)$ is some polynomial.
\end{conjecture}

\begin{conjecture}
\label{conj:mystery1}
For any $f(x)$ and $p$, the generating function $g_{f(x),p}(z)$ satisfies the equation $r(z)g_{f(x),p}(z)=r(z^p)g_{f(x),p}(z^p)+b(z)$ for some polynomials $r(z)$ and $b(z)$ depending on $f(x)$ and $p$.
\end{conjecture}

\subsection{Limits of $\displaystyle\frac{a(n)}{n^2}$}
\label{sec:Limits}

Using the generating functions, we can find the asymptotic behavior of $a(n)$ as $n$ goes to infinity. Inspired by the quadratic nature of Theorem \ref{thm:closedsimple}, we examine the behavior of $\frac{a(n)}{n^2}$.

\begin{theorem}
\label{thm:limit1xp}
For $f(x)=1+x$ and any prime $p\ge3$,
$$\lim_{n\to\infty}\frac{a_{1+x,p}(n)}{n^2}=\begin{cases}
\displaystyle\frac{p^2(p-5)(p-1)}{2(p+1)}\Big(x+\frac{p+1}{p(p-5)}\Big)^2+\frac{(p-1)(p^2-7p+4)}{2(p-5)}&\frac1p\leq x \leq\frac13
\\[1em]\displaystyle\frac{-(p-1)(7p^3-8p^2-9p+18)}{4(p+1)}\Big(x-\frac{(p+1)(3p^2-7p+6)}{7p^3-8p^2-9p+18}\Big)^2
\\[0.5em]\displaystyle+\frac{(p-1)(p^5+5p^4-8p^3-15p^2+39p-18)}{2(7p^3-8p^2-9p+18)}&\frac13\leq x \leq\frac12
\\[1em]\displaystyle\frac{(p-2)(p-1)(p^2+2p+5)}{4(p+1)}\Big(x-\frac{(p+1)^2}{p^2+2p+5}\Big)^2
\\[0.5em]\displaystyle+\frac{(p-1)(p^3+4p^2+3p-4)}{2(p^2+2p+5)}&\frac12\leq x \leq1
\end{cases}$$
where $\displaystyle n=\big\lfloor\frac{p^k}{x}\big\rfloor$ and the limit as $n\to\infty$ is with constant $x$ and $k\to\infty$.
\end{theorem}

\begin{remark}
The first polynomial from Theorem \ref{thm:limit1xp} corresponding to $\frac1p\le x\le\frac13$ should be understood in the sense of the limit for $p=5$ as we divide by $(p-5)$. In this case the polynomial is not quadratic but actually the linear polynomial $20x+8$.
\end{remark}

\begin{proof}
Theorem \ref{thm:gen1xp} states that \begin{align*}g(z)=&\sum\limits_{n\geq0}a_{1+x,p}(n)z^n\\=&\frac{1}{(1-z)^3}\bigg(1+(p-3)z+(p^2-3p+3)z^2\\&+z^2\frac{(p-1)^2}{2}\sum\limits_{i\geq0}\Big(pz^{p^i}-2(p-1)z^{2 p^i}+(p-2)z^{3 p^i}\Big)\bigg).\end{align*}
Let $\displaystyle\sum\limits_{n\geq0}b(n)z^n=\frac{z^2}{(1-z)^3}\sum\limits_{i\geq0}\Big(pz^{p^i}-2(p-1)z^{2 p^i}+(p-2)z^{3 p^i}\Big)$.
\vspace{0.1in}

Therefore, with the limit of $n=\lfloor\frac{p^k}{x}\rfloor\to\infty$ taken with fixed $x$ and $k\to\infty$, we have
\begin{align*}
\sum\limits_{n\geq0}a(n)z^n=&\frac{1+(p-3)z+(p^2-3p+3)z^2}{(1-z)^3}+\sum\limits_{n\geq0}\frac{(p-1)^2}{2}b(n)z^n\\[0.1em]
\lim_{n\to\infty}a(n)=&\lim_{n\to\infty}\left(\frac{(p-1)^2n^2}{2}+\frac{(p^2-1)n}{2}+p+\frac{(p-1)^2}{2}b(n)\right)\\[0.1em]
\lim_{n\to\infty}\frac{a(n)}{n^2}=&\frac{(p-1)^2}{2}+\frac{(p-1)^2}{2}\lim_{n\to\infty}\frac{b(n)}{n^2}.
\end{align*}

Therefore, because they act similarly, we can find the asymptotics of $\frac{a(n)}{n^2}$ by understanding the behavior $\frac{b(n)}{n^2}$. 
We can rewrite $\sum\limits_{n\geq0} b(n)z^n$ as \begin{align*}\sum_{n=0}^\infty \bigg(p\sum\limits_{i=0}^{p^i\leq n}\frac{(n-p^i)(n-p^i-1)}{2}-2(p-1)\sum\limits_{i=0}^{2p^i\leq n}\frac{(n-2p^i)(n-2p^i-1)}{2}\\+(p-2)\sum\limits_{i=0}^{3p^i\leq n}\frac{(n-2p^i)(n-2p^i-1)}{2}z^n\bigg).\end{align*} 
From this we see that \begin{align*}b(n)=&p\sum\limits_{i=0}^{p^i\leq n}\left(\frac{(n-p^i)(n-p^i-1)}{2}\right)-2(p-1)\sum\limits_{i=0}^{2p^i\leq n}\left(\frac{(n-2p^i)(n-2p^i-1)}{2}\right)\\&+(p-2)\sum\limits_{i=0}^{3p^i\leq n}\left(\frac{(n-2p^i)(n-2p^i-1)}{2}\right).\end{align*}
Therefore \begin{align*}\frac{b(n)}{n^2}=&\frac{p}{2}\sum\limits_{i=0}^{p^i\leq n}\left((1-\frac{p^i}{n})(1-\frac{p^i+1}{2})\right)-(p-1)\sum\limits_{i=0}^{2p^i\leq n}\left((1-\frac{2p^i}{n})(1-\frac{2p^i+1}{2})\right)\\&+\frac{p-2}{2}\sum\limits_{i=0}^{3p^i\leq n}\left((1-\frac{3p^i}{n})(1-\frac{3p^i+1}{2})\right).\end{align*}
Let $n=\lfloor\frac{p^k}{x}\rfloor$.
We can neglect the 1 in the second factor (it creates a change that goes to zero as $k\to\infty$), so we get $$\frac{b(n)}{n^2}=\frac{p}{2}\sum\limits_{i=0}^{p^i\leq n}(1-\frac{p^i}{n})^2-(p-1)\sum\limits_{i=0}^{2p^i\leq n}(1-\frac{2p^i}{n})^2+\frac{p-2}{2}\sum\limits_{i=0}^{3p^i\leq n}(1-\frac{3p^i}{n})^2.$$
Note that if $x \not\in [\frac13,1]$ then there is $m\in\mathbb{Z}$ such that $p^mx\in[\frac1p,1]$, so we can assume $\frac1p\leq x \leq 1$.
Ignoring the floor for simplicity, we set $n=\frac{p^k}{x}$.
Therefore we get \begin{align*}\frac{b(\frac{p^k}{x})}{(\frac{p^k}{x})^2}=\frac{p}{2}\sum\limits_{i=0}^{p^i\leq \frac{p^k}{x}}(1-p^{i-k}x)^2-(p-1)\sum\limits_{i=0}^{p^i\leq \frac{p^k}{2x}}(1-2p^{i-k}x)^2+\frac{p-2}{2}\sum\limits_{i=0}^{p^i\leq \frac{p^k}{3x}}(1-3p^{i-k}x)^2.\end
{align*}

When examining the upper limits of the three sums, we find that we therefore have 3 cases: $\frac1p\leq x \leq \frac13, \frac13\leq x \leq\frac12, \frac12\leq x \leq 1$. 
For the first sum, $p^i\leq \frac{p^k}{x}$ gives $i\leq k+1$ for $x=\frac1p$, and $i\leq k$ for $x=\frac13,\frac12,1$. 
For the second sum, $p^i\leq \frac{p^k}{2x}$ gives $i\leq k$ for $x=\frac1p,\frac12,\frac13$ and $i\leq k-1$ for $x=1$. 
For the third sum, $p^i\leq \frac{p^k}{3x}$ gives $i\leq k$ for $x=\frac1p,\frac13$ and $i\leq k-1$ for $x=\frac12,1$. 
Note that the limit is taken along the subsequences of the form $\lfloor\frac{p^k}{x}\rfloor$ with fixed $x$ and $k\to\infty$. 
Also note that the limiting function does not change if $x$ is replaced by $p\cdot x$.

For the first case of $\frac1p\leq x \leq\frac13$, we find that
\begin{align*}
\frac{b(\frac{p^k}{x})}{(\frac{p^k}{x})^2}=&\frac{p}{2}\sum\limits_{i=0}^{k}(1-p^{i-k}x)^2-(p-1)\sum\limits_{i=0}^{k}(1-2p^{i-k}x)^2+\frac{p-2}{2}\sum\limits_{i=0}^{k}(1-3p^{i-k}x)^2
\\=&(p-5)\frac{p^2-\frac{1}{p^{2k}}}{p^2-1}x^2+2\frac{p-\frac{1}{p^k}}{p-1}x
\\[0.3em]\lim_{n\to\infty}\frac{b(n)}{n^2}=&\lim_{k\to\infty}\left((p-5)\frac{p^2-\frac{1}{p^{2k}}}{p^2-1}x^2+2\frac{p-\frac{1}{p^k}}{p-1}x\right)
\\[0.3em]\lim_{n\to\infty}\frac{a(n)}{n^2}=&\frac{p^2(p-5)(p-1)}{2(p+1)}\Big(x+\frac{p+1}{p(p-5)}\Big)^2+\frac{(p-1)(p^2-7p+4)}{2(p-5)}
\end{align*}

For the case of $\frac13\leq x \leq\frac12$ we similarly find that because $$\frac{b(\frac{p^k}{x})}{(\frac{p^k}{x})^2}=\frac{p}{2}\sum\limits_{i=0}^{k}(1-p^{i-k}x)^2-(p-1)\sum\limits_{i=0}^{k}(1-2p^{i-k}x)^2+\frac{p-2}{2}\sum\limits_{i=0}^{k-1}(1-3p^{i-k}x)^2,$$
the limit of \begin{align*}\lim_{n\to\infty}\frac{a(n)}{n^2}=&\frac{-(p-1)(7p^3-8p^2-9p+18)}{4(p+1)} \bigg(x-\frac{(p+1)(3p^2-7p+6)}{7p^3-8p^2-9p+18}\bigg)^2-\frac{(p-4)(p-1)^2}{4}\\&+\frac{(p+1)(p-1)(3p^2-7p+6)^2}{4(7p^3-8p^2-9p+18)}.\end{align*}
Similarly for the case of $\frac12\leq x \leq1$ we find that because $$\frac{b(\frac{p^k}{x})}{(\frac{p^k}{x})^2}=\frac{p}{2}\sum\limits_{i=0}^{k}(1-p^{i-k}x)^2-(p-1)\sum\limits_{i=0}^{k-1}(1-2p^{i-k}x)^2+\frac{p-2}{2}\sum\limits_{i=0}^{k-1}(1-3p^{i-k}x)^2,$$
one has \begin{align*}\lim_{n\to\infty}\frac{a(n)}{n^2}=\frac{(p-2)(p-1)(p^2+2p+5)}{4(p+1)}\Big(x-\frac{(p+1)^2}{p^2+2p+5}\Big)^2+\frac{(p-1)(p^3+4p^2+3p-4)}{2(p^2+2p+5)}.\end{align*}

\end{proof}

\begin{corollary}
\label{corr:asym1xp}
For the polynomial $1+x$ and $p\ge3$,  
\begin{align*}
\liminf_{n\to\infty}\frac{a_{1+x,p}(n)}{n^2}=&\frac{(p-1)(p^3+4p^2+3p-4)}{2(p^2+2p+5)}\\[0.5em]
\limsup_{n\to\infty}\frac{a_{1+x,p}(n)}{n^2}=&\frac{(p-1)(p^5+5p^4-8p^3-15p^2+39p-18)}{2(7p^3-8p^2-9p+18)}
\end{align*}
\end{corollary}

\begin{proof}
The maximum of Theorem \ref{thm:limit1xp} is when $x=\displaystyle\frac{3p^3-4p^2-p+6}{7p^3-8p^2-9p+18}$ and the minimum is when $x=\displaystyle\frac{p^2+2p+1}{p^2+2p+5}$.
\end{proof}

We can also apply this to other $a_{f(x),p}(n)$. 

\begin{theorem}
\label{thm:limit1xx22}
For polynomial $1+x+x^2$ and prime $2$,  
$$\lim_{n\to\infty}\frac{a_{1+x+x^2,2}(n)}{n^2}=\begin{cases}
\displaystyle\frac54+\frac12x-\frac{5}{12}x^2&\frac12\leq x \leq\frac23\\[0.75em]
\displaystyle\frac32-\frac14x+\frac{7}{48}x^2&\frac23\leq x \leq1
\end{cases}$$
Furthermore, the upper and lower limits of $\frac{a_{1+x+x^2}(n)}{n^2}$ are $\displaystyle \frac{7}{5}$ and $\displaystyle\frac{39}{28}$ respectively.
\end{theorem}

The proof of Theorem \ref{thm:limit1xx22} is similar to the proof of Theorem \ref{thm:limit1xp}.

Using the recursion relations, we computed the upper and lower limits of $\frac{a_{f(x),p}(n)}{n^2}$ for sufficiently large $n$ for several $f(x)$ and $p$, 
The oscillatory nature of this sequence for large $n$ stabilizing to a periodic function in $log(x)$ is illustrated by Figure \ref{fig:ratio}.
\begin{figure}[h]
\begin{center}
\subfigure[$1+x\pmod{3}$]{
\includegraphics[width=0.3\textwidth]{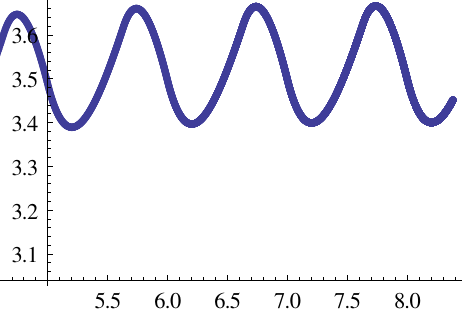}
}
\subfigure[$1+x\pmod{5}$]{
\includegraphics[width=0.3\textwidth]{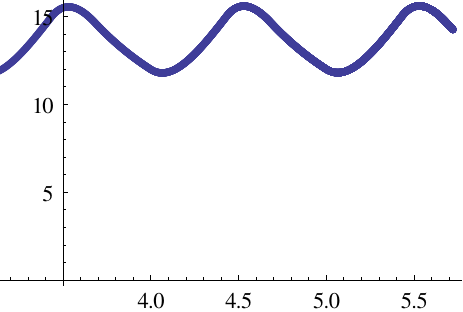}
}
\subfigure[$1+x+x^2\pmod{2}$]{
\includegraphics[width=0.3\textwidth]{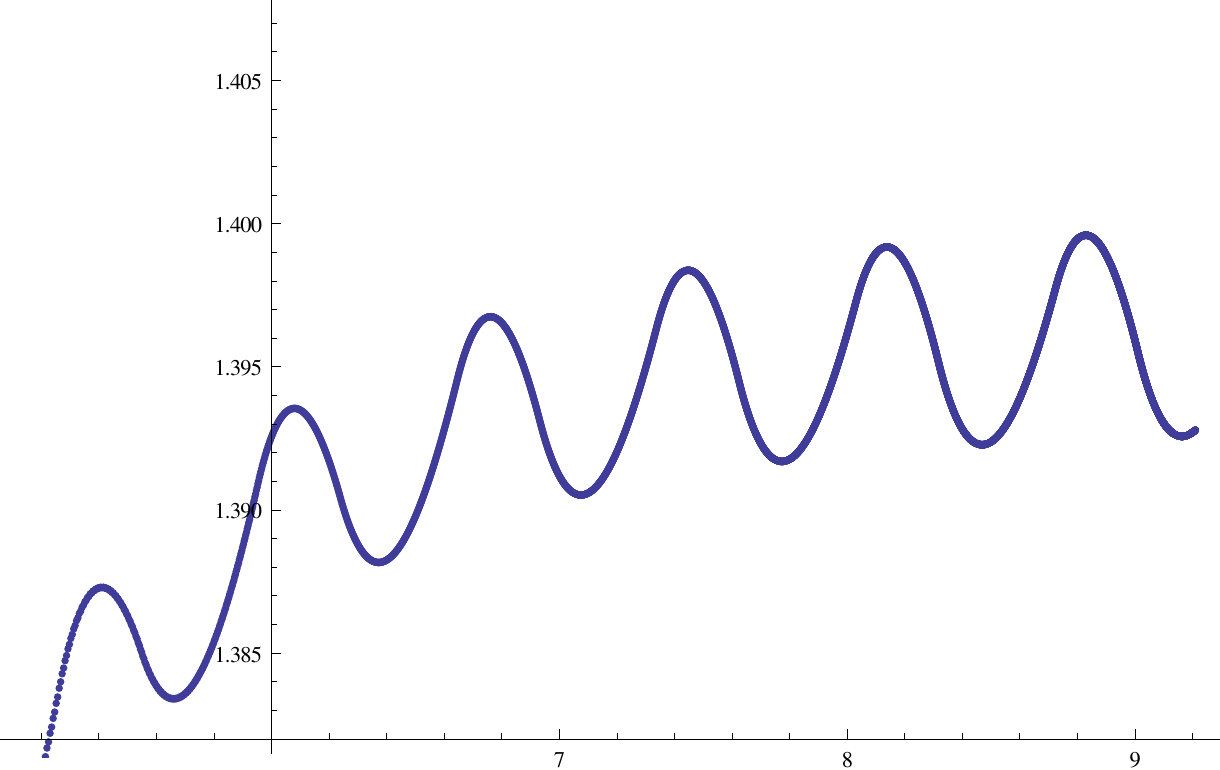}
}
\end{center}
\caption{$\displaystyle\frac{a_{f(x),p}(m)}{m^2}$ with the $x$ axis showing log$_pm$}
\label{fig:ratio}
\end{figure}

This matches a prior result expressed in Lemma 5.15 by \cite{AB}, which states that for large $n$, there exists constants $c_1$ and $c_2$ such that $c_1 n^2 \leq a(n) \leq c_2 n^2$. 
The limits given by Corollary \ref{corr:asym1xp} provide sharp values of $c_1$ and $c_2$. 
\footnote{Strictly speaking, fore these sharp values, we may not have $c_1n^2\le a(n)\le c_2n^2$, but for any $\delta>0$ we have $(c_1-\delta)n^2\le a(n)\le(c_2+\delta)n^2$ for large enough $n$.}

\section{Counting Coefficients}
\label{sec:Matrix}

\subsection{Definitions}
\label{sec:MatrixDef}

For a polynomial $f(x)$, prime $p$, and positive integer $\alpha \leq p-1$, we define $q_{f(x),p}(k,\alpha)$ to be the number of occurrences of $\alpha$ among the coefficients of $f(x)^k$ reduced modulo $p$. Similarly, we define $q_{f(x),p}(k)$ to be the total number of nonzero coefficients of $f(x)^k$. We then define $r_{f(x),p}(n,\alpha)=\sum\limits_{i=0}^{n-1}q_{f(x),p}(i,\alpha)$ and $r_{f(x),p}(n)=\sum\limits_{i=0}^{n-1}q_{f(x),p}(i)$. We search for a quick method for calculating both $q_{f(x),p}(k,\alpha)$ and the asymptotic behavior of $r_{f(x),p}(n,\alpha)$ for large $n$.

\subsection{Willson Method}
\label{sec:Willson}

Willson \cite{W1} describes an algorithm for computing the value of $r_{f(x),2}(n)$, which is provided in Theorem \ref{thm:Willson}.

\begin{theorem}[Willson's Method]
\label{thm:Willson}
For some polynomial $f(x)$ with maximum degree $d$, there exists a matrix $B$, row vector $u$, and column vector $v$ each of size $2^d-1$ such that $u \cdot B^k \cdot v=r_{f(x),2}(2^k)$. 
\end{theorem}

Amdeberhan-Stanley \cite{ASt} describes a similar and related algorithm for calculating the number of each coefficient $\alpha$ for any power $k$ for general $f(x)$ and $p$, namely $q_{f(x),p}(k,\alpha)$. Willson also analyzed the case of $p>2$ in \cite{W2}.

\begin{example}
\renewcommand\arraystretch{.5}
For $1+x+x^2$ mod $2$, $B=\begin{bmatrix}2&0&2\\1&1&2\\1&1&0\end{bmatrix}$. Note that the largest eigenvalue of this matrix is $1+\sqrt{5}$.
\renewcommand\arraystretch{1}
\end{example}

\begin{theorem}
\label{thm:transcomp}
The matrix $B$ is the sum of four matrices, each of which corresponds to a self-mapping of the set $X=F_2[x]/x^d \setminus {0}$.
\end{theorem}

Theorem \ref{thm:transcomp} follows easily from Willson \cite{W1}.

\begin{remark}
The size of the matrix $B$ can be made smaller only by using accessible blocks, as explained in Wilson \cite{W1}.
\end{remark}

\subsection{Eigenvalue Analysis}
\label{sec:Eigen}

The matrix $B$ has nonnegative entries and is irreducible. Following Willson \cite{W1}, define $\lambda$ to be the Perron-Frobenius eigenvalue of $B$, i.e., the largest positive eigenvalue of $B$ (it exists by the Perron-Frobenius theorem).
We define $\lambda(f)$ to be the value of $\lambda$ for the polynomial $f(x)$. 
We can approximate the value of $r_{f(x),p}(p^k,\alpha)$ with $\lambda^k$ because the entries of $B^k$ grow as a constant times $\lambda^k$. 

\begin{example}
For $f(x)=1+x$ and $p=2$, $\lambda=3$ because $B=[3]$. In this case $\lambda$ corresponds exactly to the scaling of the number of nonzero coefficients when doubling the number of rows, namely $r_{1+x,2}(2k)=3\cdot r_{1+x,2}(k)$.
\end{example}

When examining the eigenvalues, we note that there are multiple transformations of a polynomial that does not change $\lambda$.

\begin{theorem}
\label{thm:eigenequiv}
We define the polynomials $f(x)$ and $g(x)$ to be similar if we can transform $f(x)$ into $g(x)$ through a combination of the transformations $f(cx) \text{ and } cf(x)$ with integer $1<c<p$, $x^cf(x)$  with integer $c>0$, $f(x^c)$ with integer $c>1$, $x^{deg(f)}f(x^{-1})$, and $f(x)^c$ with integer $c>1$. Any two similar polynomials have the same $\lambda$.
\end{theorem}

\begin{proof}
Because the transformations $f(c \cdot x), f(x^c), x^c\cdot f(x), c \cdot f(x)$, and flipping a polynomial do not change the number of nonzero coefficients of a polynomial, $\lambda$ do not change. 
Furthermore, because $f(x)^c$ is every $c^{th}$ row, the ratios over the long term of the sums of total number of nonzero coefficients does not change, so $\lambda$ is the same. 
Namely, let $q_{f(x)}(n)$ be the number of nonzero coefficients of $f(x)^n$. 
Therefore $q_{f(x)}(n+1)\le C\cdot q_{f(x)}(n)$, where $C$ is the number of nonzero coefficients of $f(x)$. 
This means that \begin{align*}r_{f(x)}(k\cdot n)=\sum\limits_{j=0}^{k\cdot n -1} q_{f(x)}(j) \le \sum\limits_{j=0}^{n-1}(1+C+\ldots+C^{k-1})q_{f(x)}(j\cdot k) \le (1+C+\ldots+C^{k-1})r_{f(x)^k}(n). \end{align*}
This implies that $\lambda(f)\le\lambda(f^k)$. 
Similarly since $q_{f(x)}(j\cdot k -i) \ge C^{-i}q_{f(x)}(j\cdot k)$, we can show that $\lambda(f^k)\le\lambda(f)$. 
Therefore $\lambda(f)=\lambda(f^k)$.
\end{proof}

\subsubsection{Values of $\lambda$ where $p=2$}
We calculate $\lambda$ for polynomials with $p=2$. We also find the minimal polynomial of $\lambda$. Provided are $\lambda$ and the degree $d$ of its minimal polynomial for non-similar polynomials with degree of up to 6, although we had calculated for $deg(f)\le9$.

\begin{table}
\begin{center}
\begin{tabular}{|ccc|ccc|}
\hline
Polynomial					  &$ \lambda  $&$d  $&  Polynomial					  &$\lambda   $&$d  $\\\hline
$ 1+x 			 	   	$&$ 3		$&$1  $&$1+x+x^6					$&$ 3.45686 $&$20$\\
$ 1+x+x^2					$&$ 3.23607 	$&$2  $&$1+x+x^2+x^6				$&$ 3.49009 $&$20$\\
$ 1+x+x^3					$&$ 3.31142 	$&$4  $&$1+x+x^3+x^6				$&$ 3.50478 $&$10$\\
$ 1+x+x^4					$&$ 3.33159 	$&$5  $&$1+x^2+x^3+x^6				$&$ 3.53521 $&$20$\\
$ 1+x+x^2+x^4				$&$ 3.3788  	$&$7  $&$1+x+x^2+x^3+x^6			$&$ 3.53141 $&$19$\\
$ 1+x+x^3+x^4				$&$ 3.47662 	$&$4  $&$1+x+x^4+x^6				$&$ 3.50468 $&$17$\\
$ 1+x+x^2+x^3+x^4			$&$ 3.45729 	$&$4  $&$1+x+x^2+x^4+x^6			$&$ 3.55002 $&$19$\\
$ 1+x+x^5					$&$ 3.35174	$&$10$&$1+x+x^3+x^4+x^6			$&$ 3.59415 $&$16$\\
$ 1+x^2+x^5				$&$ 3.46127 	$&$12$&$1+x^2+x^3+x^4+x^6			$&$ 3.53665 $&$15$\\
$ 1+x+x^2+x^5				$&$ 3.49563 	$&$7  $&$1+x+x^2+x^3+x^4+x^6		$&$ 3.59043 $&$11$\\
$ 1+x+x^3+x^5				$&$ 3.45469 	$&$12$&$1+x+x^5+x^6				$&$ 3.54536 $&$14$\\
$ 1+x^2+x^3+x^5			$&$ 3.46639 	$&$5  $&$1+x+x^2+x^5+x^6			$&$ 3.50809 $&$18$\\
$ 1+x+x^2+x^3+x^5			$&$ 3.5229   	$&$14$&$1+x+x^2+x^3+x^5+x^6		$&$ 3.57066 $&$17$\\
$ 1+x+x^2+x^4+x^5			$&$ 3.47168 	$&$11$&$1+x+x^2+x^4+x^5+x^6		$&$ 3.49995 $&$6  $\\
$ 1+x+x^2+x^3+x^4+x^5		$&$ 3.52951 	$&$6  $&$1+x+x^2+x^3+x^4+x^5+x^6	$&$ 3.5598   $&$6  $\\
\hline
\end{tabular}
\caption{$\lambda$ and the degree of its minimal polynomial for $p=2$ and deg$(f(x)) \leq 6$}
\label{tbl:eigenvalues}
\end{center}
\end{table}
We see that $\lambda$ is between $3$ and $4$. We form several conjectures on the bounds of $\lambda$.

\begin{conjecture}
When $p=2$, $\lambda \geq 3$. Furthermore, $\lambda=3$ only for polynomials similar to $1+x$. If $p=2$ and $\lambda > 3$, then $\lambda \geq 1+\sqrt{5}$. Furthermore, $\lambda=1+\sqrt{5}$ only if $f(x)$ is similar to $1+x+x^2$.
\end{conjecture}

\begin{question}
Is it true that $\lambda(f)=\lambda(g)$ if and only if $f(x)$ and $g(x)$ are similar in terms of the transformations described in Theorem \ref{thm:eigenequiv}?
\end{question}

\begin{theorem}
\label{eigenlimit}
For some polynomial $f(x)$ with degree at most $2^k$ and $p=2$, \\$\lambda(f) \leq 4(1-\frac{1}{2^{k+2}})^\frac{1}{k+1}$.
\end{theorem}

\begin{proof}
Define $k$ such that the degree of $f(x)$ is at most $2^k$, with $p=2$. From Theorem \ref{thm:transcomp}, we can draw an oriented graph whose vertices are elements of $X$ and whose edges correspond to the four maps. Therefore there are exactly four edges coming out of each vertex. Therefore if $Q(n)$ is the number of paths in the graph of length $n$, we have $ \displaystyle \log\lambda=\limsup_{n\to\infty}\frac{\log Q(n)}{n}$. From the definition of Willson's method, Theorem \ref{thm:Willson}, two of the four mappings correspond to $g(x) \to g(x^2)$ and $g(x) \to x\cdot g(x^2)$. Assume deg$(f(x))=2^k$. Then a path starting from any $g(x)$ and moving first to $x\cdot g(x^2)$ then alternating in any way between the two mappings leads to $0$ after $k+1$ steps. So the number of such paths of length $k+1$ is $2^k$. So the number of paths of length $k+1$ from any point that avoids $0$ is at most $4^{k+1}-2^k$. Thus the number of such paths of length $n\cdot(k+1)$ is at most $(4^{k+1}-2^k)^n$. This gives us the bound of $\lambda \leq 4(1-\frac{1}{2^{k+2}})^\frac{1}{k+1}$.
\end{proof}

For $k=0$, the only polynomial is $1+x$, so the bound $\lambda \leq 4(1-\frac{1}{4})^1=3$ is sharp. However, for $k=1$ the bound tells us that $\lambda \leq \sqrt{14}$ which is not sharp. Furthermore, this bound approaches $4$ as $k$ approaches $\infty$.

\begin{conjecture}
Let $\Lambda_k$ be the maximal $\lambda(f)$ for deg $f\leq k$. Then $lim_{k\to\infty} \Lambda_k=4$.
\end{conjecture}
\begin{remark} 
Similarly for $p>2$, one may conjecture that  $lim_{k\to\infty} \Lambda_k=p^2$.
\end{remark}
Through computer analysis of $\lambda$ for $p=2$ and $deg\Big(f(x)\Big)\leq9$, Conjecture \ref{conj:lambdadegree} arises.

\begin{conjecture}
\label{conj:lambdadegree}
The degree of the minimal polynomial of $\lambda$ is less than or equal to $2^{deg(f)-1}$ for $p=2$.
\end{conjecture}

\section{Conclusion and Directions of Future Research}
\label{sec:Conclusion}

Natural goals for further study of the phenomena examined in this paper include the following:
\begin{itemize}
\item Obtain recursion relations, generating functions, and limiting functions as in Section 2 for $a_{f(x),p}(n)$ in the case $deg\Big(f(x)\Big)>1$;
\item Prove Conjecture \ref{conj:mystery1} on the functional equation for the generating function for $a_{f(x),p}(n)$;
\item Prove the conjectures in section 3 on the behavior of the eigenvalues $\lambda$ and obtain better upper bounds;
\item Find, tighten, and explore the upper bound mentioned in Conjecture \ref{conj:lambdadegree};
\item Study the algebras generated by the four transformations composing the Willson matrices and find analogs for larger $p$.
\end{itemize}

\section{Acknowledgments} 

Thanks go to the Center for Excellence in Education, the Research Science Institute, and the Massachusetts Institute of Technology for the opportunity to work on this project.
I would also like to thank Dr.\ Pavel Etingof for suggesting and supervising the project and Dorin Boger for mentoring the project. 
I would also like to thank RSI head mentor Tanya Khovanova for many useful discussions, ideas, suggestions, and feedback and Dr.\ John Rickert for feedback on the paper.
Finally, I would like to give thanks to Informatica, The Milken Family Foundation, and the Arnold and Kay Clejan Charitable Foundation for their sponsorship.

\begin{singlespace}
\bibliographystyle{abbrv}

\bibliography{biblio}
\end{singlespace}

%

\appendix


%

\end{document}